\documentclass[12pt]{article}
\usepackage{epsfig}
\usepackage{amsfonts}
\usepackage{amssymb}
\usepackage{amsmath}
\usepackage{amsthm}
\usepackage{latexsym}
\usepackage{graphicx}
\usepackage{bm}
\usepackage{tabularx}
\usepackage{booktabs} 
\usepackage{enumerate}
\usepackage{caption}
\usepackage[titletoc]{appendix}
\usepackage[numbers]{natbib}
\usepackage{bbm}
\usepackage{url}
\usepackage{hyperref}
\usepackage{color}
\usepackage{xcolor}
\usepackage[section]{placeins}
\usepackage{subcaption}
\usepackage{subfloat}
\usepackage{placeins}

\usepackage{floatrow}
\newfloatcommand{capbtabbox}{table}[][\FBwidth]

\usepackage{blindtext}

\let\Oldsection\section
\renewcommand{\section}{\FloatBarrier\Oldsection}

\let\Oldsubsection\subsection
\renewcommand{\subsection}{\FloatBarrier\Oldsubsection}

\let\Oldsubsubsection\subsubsection
\renewcommand{\subsubsection}{\FloatBarrier\Oldsubsubsection}

\catcode`\@=11

\newcommand{\shortdot}[1]{\raisebox{-0.4pt}{$\stackrel{\bullet}{#1}$}}

\theoremstyle{plain}
\newtheorem{theorem}{Theorem}[section]
\newtheorem{lemma}[theorem]{Lemma}

\theoremstyle{definition}

\theoremstyle{remark}
\newtheorem*{remark}{Remark}

\begin{document}
\title{Queues with Updating Information: Finding the Amplitude of Oscillations}
\author{ Philip Doldo \\ Center for Applied Mathematics \\ Cornell University
\\ 657 Rhodes Hall, Ithaca, NY 14853 \\  pmd93@cornell.edu  \\ 
\and
  Jamol Pender \\ School of Operations Research and Information Engineering \\ Center for Applied Mathematics \\ Cornell University
\\ 228 Rhodes Hall, Ithaca, NY 14853 \\  jjp274@cornell.edu 
 }    

\maketitle

\begin{abstract} 
Many service systems provide customers with information about the system so that customers can make an informed decision about whether to join or not.  Many of these systems provide information in the form of an update.  Thus, the information about the system is \textbf{updated} periodically in increments of size $\Delta$.  It is known that these updates can cause oscillations in the resulting dynamics.  However, it is an open problem to explicitly characterize the size of these oscillations when they occur.  In this paper, we solve this open problem and show how to exactly calculate the amplitude of these oscillations via a fixed point equation.  We also calculate closed form approximations via Taylor expansions of the fixed point equation and show that these approximations are very accurate, especially when $\Delta$ is large.  Our analysis provides new insight for systems that use updates as a way of disseminating information to customers.     
\end{abstract}


\section{Introduction} \label{sec_intro}

In many queueing systems, customers are provided information about the queue length or waiting time so customers can make an appropriate decision about whether to join or not.  This queue length information has the potential to impact a service system in a variety of ways.  However, it is often that this queue length or waiting time information is not provided in real-time.  One cause of this delay in information is that it often takes time to process and push the information to customers.  Much of the delayed information literature focuses on delays that are constant throughout time, see for example  \citet{ mitzenmacher2000useful,  lipshutz2015existence, lipshutz2017exit, pender2017queues, pender2017strong, pender2018analysis, Lipshutz2018, nirenberg2018impact, novitzky2020limiting}.  

However, in practice many of these delays are caused by systems that update periodically.  Periodic updates help service systems managers balance the cost of information and computational issues into consideration.  Unfortunately, there is not much research on queueing systems with updates and in fact the only paper that considers such models is \citet{novitzky2020update}.  Moreover, in \citet{novitzky2020update} the authors prove a functional strong law of large numbers limit showing that an appropriately scaled queueing process converges to a functional differential equation (FDE) system.  With the limiting FDE system, they show that the system can undergo a Hopf bifurcation when the updating interval $\Delta$ is sufficiently large.  However, there is some research from the dynamical systems community on control systems with updates.  In the control literature, these types of updates are called piecewise continuous arguments or delays \citet{silkowski1979star, cooke1984retarded, aftabizadeh1987differential, wiener1989oscillations, cooke1991survey}.  Not surprisingly, much of the research centers around determining the stability regions for these types of systems to understand when oscillations will or will not occur.  

However, one important open problem that remains is to find the amplitude of the oscillations that result from the Hopf bifurcations.  Instability in a queueing system causes queue lengths to oscillate with time which can lead to inefficienies due to some servers being overworked and others being underworked.  Being able to compute the amplitude of the queue length oscillations could help a service manager quantify the level of ineffiency present due to the system's instability.  Typically oscillations in systems with delays are calculated approximately by a method called Lindstedt's method.  In fact, how to approximate the amplitude of oscillations in queueing systems with delays is outlined in \citet{novitzky2019nonlinear}.  Despite the analysis carried out in \citet{novitzky2019nonlinear} it still remains an open question to compute the amplitude of oscillations in the FDE model we study in this paper.  The main reason is that the delay is \textbf{non-stationary} with respect to time, which renders the Lindstedt approach ineffective.  A new approach must be derived, which is the main focus of this work. Thus, in this paper, we answer the following question: when oscillations occur, how does the amplitude of the oscillations depend on the model parameters?

 \subsection{Main Contributions of Paper}

The contributions of this work can be summarized as follows:    
\begin{itemize}
\item We develop a functional differential equation model for queues with updating information.
\item We derive a fixed-point equation for computing the steady-state amplitude for our updating queueing system when the system undergoes a Hopf bifurcation in the two-dimensional case.  
\item We derive new closed-form approximations for the amplitude using first order and second order Taylor expansions and show that the amplitude can be upper bounded by the first order Taylor expansion in the two-dimensional case.
\item We derive nonlinear equations that can be solved to compute the steady-state amplitude when the system undergoes a Hopf bifurcation in the $N$-dimensional case for $N > 2$.
\item We derive closed-form approximations of the amplitudes using first-order Taylor expansions in the $N$-dimensional case for $N > 2$.
\end{itemize} 


\subsection{Organization of Paper}

The remainder of this paper is organized as follows.   In Section \ref{updating_queueing_model_section} we introduce the updating queueing model, determine how to compute the steady-state amplitude of the queue length oscillations in the two-dimensional setting, and we introduce linear and quadratic closed-form approximations of the steady-state amplitude. In Section \ref{multidimensional_case_section}, we consider finding the steady-state amplitudes of queue length oscillations when we have $N > 2$ queues in our system. We consider separately the cases when $N$ is even and when $N$ is odd and in each case we compare the result with closed-form approximations of the steady-state amplitudes. Finally, in Section \ref{conclusion}, we give concluding thoughts and discuss potential ideas for future research.


\section{Updating Queueing Model}
\label{updating_queueing_model_section}

In this section, we present a functional dynamical system queueing model where we have $N$ queues operating in parallel and customers choose which station to join via a customer choice model that depends on the queue length.  We assume that the total arrival rate to the system (sum of all queues) is $\lambda$, the service rate for each of the infinite number of servers at each queue is given by $\mu$.  In the spirit of delayed information, customers do not observe the real-time queue length.  However, customers observe the queue length at the time of the most current update, which are periodic with size $\Delta$.  Thus, the information that the customer receives is actually the queue length $t - \left \lfloor \frac{t}{ \Delta} \right \rfloor \Delta $ time units in the past,  The function $t - \left \lfloor \frac{t}{ \Delta} \right \rfloor \Delta $ is also known as a sawtooth function.  Thus, the customer will not make their decision on which queue to join based on the real-time queue length $q(t)$, but they will make their decision based on $q\left(\left \lfloor \frac{t}{ \Delta} \right \rfloor \Delta \right)$, which is precisely the queue length at the time of the previous update.   

One important thing to note is that the sawtooth function $t - \left \lfloor \frac{t}{ \Delta} \right \rfloor \Delta $ \textbf{is not constant} like in a constant delay model and increases linearly within an updating interval.  Thus, if we interpret the update as a delay in information, then the delay that the customer experiences is \textbf{non-stationary} and is not constant.  This is an important distinction from the constant delay models studied in \citet{ pender2017queues, pender2018analysis} and the non-stationary behavior of the periodic updates will add additional complexities to our analysis of the updating model.  We will also show that we cannot rely on previous analytical methods that were developed in \citet{pender2017strong} for the constant delay setting because of the non-stationarity of the delay function.  Non-stationary delay FDE models have been explored in the dynamical systems and control theory literature, see for example \citet{louisell2001delay, niculescu1998robust, cooke1991survey}.  However, much of this literature focuses on either looking at the average variation of the time varying function or providing robust bounds for stability with a non-stationary delay function.    

The $N$-dimensional queueing model that our work focuses on is given by the following system of functional differential equations.  
\begin{eqnarray} \label{fdesN}
\overset{\bullet}{q}_i(t) = \lambda \cdot \frac{\exp \left( - \theta \cdot q_i(\Phi(t, \Delta)) \right) }{\sum_{j=1}^{N} \exp\left( -\theta \cdot q_j( \Phi(t, \Delta) \right)} - \mu q_i(t), \hspace{5mm} i = 1, ..., N
\end{eqnarray}

\noindent where $$\Phi(t, \Delta) := \bigg\lfloor \frac{t}{\Delta} \bigg\rfloor \Delta.$$ In this model, the parameter $\lambda > 0$ represents the arrival rate into the queueing system, $\mu > 0$ represents the service rate of the system, and $q_i(t)$ represents the length of the $i^{\text{th}}$ queue at time $t$. We use a multinomial logit model to model customer choice which is informed by information from the most recent update time. That is, $$\frac{\exp \left( - \theta \cdot q_i(\Phi(t, \Delta)) \right) }{\sum_{j=1}^{N} \exp\left( -\theta \cdot q_j( \Phi(t, \Delta) \right)}$$ can be interpreted as the probability that a customer joins the $i^{\text{th}}$ queue, where we note that this probability depends on the queue lengths at time $\Phi(t, \Delta)$, which is the most recent update time that occurred at or before time $t$, so customers are only informed by queue length information from this time which will update once $\big\lfloor \frac{t}{\Delta} \big\rfloor$ changes value. Stability properties of the model were considered in \citet{novitzky2020update} where it was found that this system undergoes a Hopf bifurcation at the critical delay $$\Delta_{\text{cr}} = \frac{\ln \left( 1 + \frac{2}{ \frac{\lambda \theta}{\mu N} -1  }  \right)}{\mu}$$ provided that $\frac{\lambda}{\mu N} < 1.$ We note that as $N$ increases, so does the critical delay (assuming that the condition $\frac{\lambda}{\mu N} < 1$ is not violated). Moving forward, we will be interested in computing the steady-state amplitude of the queue length oscillations that result from the Hopf bifurcation.

\subsection{Two-Dimensional Case}
Before thinking about the more general multi-dimensional case, we explore the two-dimensional case to gain important intuition about the model.  In the two-dimensional case, our updating queueing system is given by the following functional differential system
\begin{eqnarray} \label{fdes}
\overset{\bullet}{q}_i(t) = \lambda \cdot \frac{\exp \left( - \theta \cdot q_i(\Phi(t, \Delta)) \right) }{\sum_{j=1}^{2} \exp\left( -\theta \cdot q_j( \Phi(t, \Delta) \right)} - \mu q_i(t), \hspace{5mm} i = 1, 2
\end{eqnarray}

\noindent where $$\Phi(t, \Delta) := \bigg\lfloor \frac{t}{\Delta} \bigg\rfloor \Delta.$$ 

 Thus, for $k \in \mathbb{Z}^+ \cup \{ 0 \}$ and $\Delta > 0$, we have that $$\Phi(t, \Delta) = k \Delta$$ when $t \in [k \Delta, (k+1) \Delta)$ and on this interval our updating queueing system reduces to a system of ordinary differential equations (ODEs) i.e.

\begin{eqnarray}
\overset{\bullet}{q}_i(t) = \lambda \cdot \frac{\exp \left( - \theta \cdot q_i(k \Delta) \right) }{\sum_{j=1}^{2} \exp\left( -\theta \cdot q_j( k \Delta) \right)} - \mu q_i(t), \hspace{5mm} i = 1, 2.
\end{eqnarray}

This is a nice observation because we can analyze the system on each interval of size $\Delta$ and repeat the process over and over again.  One other observation to make is to notice that on each interval of size $\Delta$ the arrival rate to each queue is fixed.  Thus, on a specific interval, one can view the queueing system as a system of infinite server queues with a constant arrival rate.  From a differential equations perspective, this reduces to linear ODEs and allows us to get a closed-form solution in each interval given the starting point.  However, before we begin to use these observations to find the amplitude for the queueing model, we find it important to recall a standard result for linear differential equations. We will exploit this result in the sequel.  

\begin{lemma} \label{ode_soln}
Let q(t) be the solution to the following differential equation
\begin{equation}
\shortdot{q} = \alpha - \beta q(t) 
\end{equation}
where q(0) = x.  Then the solution for any value of $t$ is given by 
\begin{equation}
q(t) = x e^{-\beta t} + \frac{\alpha}{\beta} \left( 1 - e^{-\beta t} \right).
\end{equation}
\begin{proof}
This follows from standard results on ordinary differential equations.  
\end{proof}
\end{lemma}

Using Lemma \ref{ode_soln}, we observe that on the interval $t \in [ k \Delta, (k+1) \Delta)$, we can solve each ODE to get 
\begin{equation}
\label{q_of_t_eqn}
q_i(t) = q_i(k \Delta) e^{-\mu (t - k \Delta)} +\rho \frac{\exp \left( - \theta \cdot q_i(k \Delta) \right)}{\sum_{j=1}^{2} \exp(-\theta \cdot q_j(k \Delta))}(1 - e^{-\mu (t - k \Delta)}).
\end{equation} 

\noindent where, for ease of notation, we have defined $$\rho := \frac{\lambda}{\mu}$$ and will continue to use this definition throughout the paper. This implies that

\begin{equation} \label{queue_delta}
q_i((k+1) \Delta) = q_i(k \Delta) e^{-\mu  \Delta} +\frac{  \rho  \exp \left( - \theta \cdot q_i(k \Delta) \right)}{\sum_{j=1}^{2} \exp(-\theta \cdot q_j(k \Delta))}(1 - e^{-\mu \Delta}).
\end{equation}


\begin{figure}[ht]
\includegraphics[scale=.4]{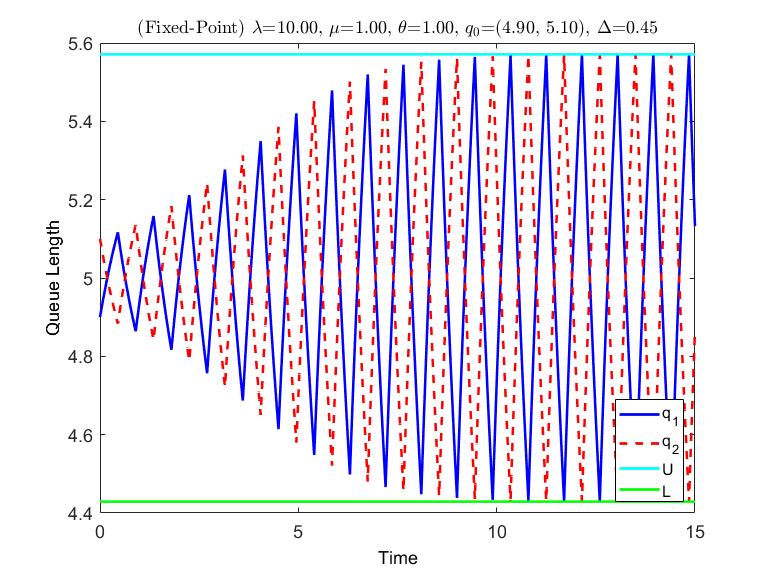}
\caption{The queue lengths plotted against time for where the horizontal lines represent our approximations of the amplitude of the oscillations. This plot is for $\lambda = 10, \mu = 1, \theta=1, \Delta=.45$, and for $t \in [- \Delta, 0]$ we have $q_1 = .49$ and $q_2 = .51.$}
\label{fig_intuition}
\end{figure}


We see from Equations \ref{q_of_t_eqn} and \ref{queue_delta} that we can describe the dynamics of the update system. We show in the following Theorem that we can exactly compute the amplitude of the oscillations in the two-dimensional case using a fixed-point equation.

\begin{theorem}
Given that $(q_1(t),q_2(t))$ solves the dynamics given in Equation \ref{fdes} and a Hopf bifurcation occurs, then the steady-state minimum, $L$, is the non-trivial solution to the following fixed-point equation

\begin{eqnarray}
\rho - L &=& L e^{-\mu \Delta} + \rho \frac{e^{-\theta \cdot L}}{e^{-\theta \cdot L} + e^{- \theta \cdot (\rho - L)}}(1 - e^{- \mu \Delta}). \label{subequation}
\end{eqnarray}
Moreover, the amplitude of oscillations is given by the formula 
\begin{equation} \label{amp}
\mathrm{Amplitude} = \frac{\rho}{2} - L.  
\end{equation}

\begin{proof}
Before we start a more formal proof, we think it is important to build some intuition for this result.  The first observation to make is that we know that as time gets large after a Hopf bifurcation that the queueing system settles down to a periodic steady state.  The second observation is that in this steady state, the first queue grows to a maximum and the second queue shrinks to a minimum.  When the time crosses another integer multiple of delta, the queues switch and the first queue drops to the minimum and the second queue increases to the maximum.  This happens indefinitely as the system of functional differential equations has undergone a Hopf bifurcation.  As a result, we will show that our queueing problem can be mapped to one a problem where a capacitor charges and subsequently discharges. 

Thus, when the update system is unstable, we have that the queue lengths will oscillate periodically and that these oscillations will approach some limiting amplitude. That is, if we let $L$ denote the greatest lower bound of the queue lengths for all time and $U$ denote the least upper bound of the queue lengths as time gets large, then for some sufficiently large $k \in \mathbb{Z}^+$ and for one of the queues (without loss of generality, let this queue be $q_1$) we have that $$q_1(k \Delta) = U$$ and $$q_1((k+1)\Delta) = L.$$

\noindent By following the differential equation given in Equation \ref{fdes} for $\Delta$ time units using Equation \ref{queue_delta}, this implies that
\begin{eqnarray}
L &=& q_1(k\Delta) e^{- \mu \Delta} + \rho \frac{e^{-\theta \cdot q_1(k\Delta)}}{e^{-\theta \cdot q_1(k \Delta)} + e^{-\theta \cdot q_1((k+1)\Delta)}} (1 - e^{- \mu \Delta}) \\
&=& U e^{- \mu \Delta} + \rho \frac{e^{-\theta \cdot U}}{e^{-\theta \cdot U} + e^{-\theta \cdot L}} (1 - e^{- \mu \Delta}). \label{Lequation}
\end{eqnarray}

\noindent Similarly, for any $k' \in \mathbb{Z}^+$ such that $k' = k + (2n+1) m$ for some $m \in \mathbb{Z}^+$, we have that $$q_1(k'\Delta) = L$$ and $$q_1( (k'+1) \Delta) = U$$ and we get that 

\begin{eqnarray}
U &=& q_1(k'\Delta) e^{- \mu \Delta} + \rho \frac{e^{-\theta \cdot q_1(k'\Delta)}}{e^{-\theta \cdot q_1(k' \Delta)} + e^{-\theta \cdot q_1((k'+1)\Delta)}} (1 - e^{- \mu \Delta})\\
&=& L e^{- \mu \Delta} + \rho \frac{e^{-\theta \cdot L}}{e^{-\theta \cdot L} + e^{-\theta \cdot U}} (1 - e^{- \mu \Delta}). \label{Uequation}
\end{eqnarray}

\noindent Thus, finally we arrive at a two-dimensional system of nonlinear equations i.e.

\begin{eqnarray}
L &=& U e^{- \mu \Delta} + \rho \frac{e^{-\theta \cdot U}}{e^{-\theta \cdot U} + e^{-\theta \cdot L}} (1 - e^{- \mu \Delta})  \label{LandU1} \\
U &=& L e^{- \mu \Delta} + \rho \frac{e^{-\theta \cdot L}}{e^{-\theta \cdot L} + e^{-\theta \cdot U}} (1 - e^{- \mu \Delta}).  \label{LandU2}
\end{eqnarray}

\noindent Now by adding Equations \ref{LandU1} and \ref{LandU2}, we observe the following relationship between the upper and lower values
\begin{eqnarray}
L  + U &=&  \rho .
\end{eqnarray}

\noindent Thus, we can use this observation to substitute the quantity $U = \rho - L$ in Equation \ref{LandU2}  to get an expression that completely depends on L.  Once one makes this substitution, one obtains the following equation
\begin{eqnarray}
\rho - L = L e^{-\mu \Delta} + \rho \frac{e^{-\theta \cdot L}}{e^{-\theta \cdot L} + e^{- \theta \cdot (\rho - L)}}(1 - e^{- \mu \Delta}). \label{subequation}
\end{eqnarray}

\noindent We can solve Equation \ref{subequation} numerically to get a value for $L$ and then use the fact that $U = \rho - L$ to get a value for $U$. Moreover, since we can determine the values of $L$ and $U$, we can determine the amplitude of the oscillations of the queue lengths by observing that the amplitude is given by the following formula
\begin{eqnarray}
\text{Amplitude} = \frac{U - L}{2} = \frac{ \left(\rho - L \right) - L}{2} = \frac{\rho}{2} - L .
\end{eqnarray}
This completes the proof.  
\end{proof}
\end{theorem}

Now that we have a fixed-point equation to derive the ampltiude of the oscillations, it is important to see numerically how this works.  In Figure \ref{fig2}, we plot two different examples of oscillating queues.  We see in both plots that our fixed-point equation captures the correct amplitude of the oscillations.  Thus, we immediately see the value of our fixed-point equation to give us insight about the size of oscillations in queues with information updates.  
 

\begin{figure}
\centering 
\includegraphics[scale=.3]{./Figures/FP_ex4_hq.jpg}~\includegraphics[scale=.3]{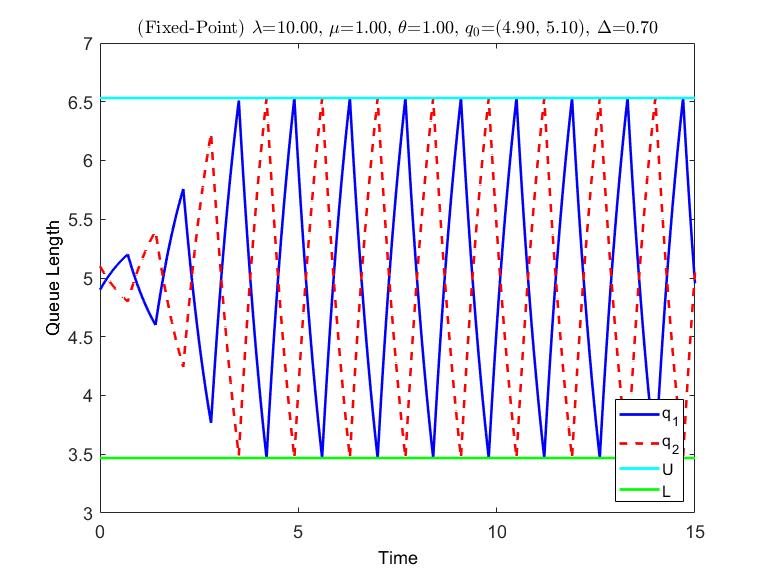}
\caption{The queue lengths plotted against time where the horizontal lines represent our approximations of the amplitude of the oscillations by numerically solving the fixed-point equation for $L$. These plots are for $\lambda = 10, \mu = 1, \theta=1$, and for $t \in [- \Delta, 0]$ we have $q_1 = .49$ and $q_2 = .51.$ In the left plot $\Delta=.45$ and in the right plot $\Delta = .7$.}
\label{fig2}
\end{figure}

Although we have a quite simple fixed-point equation which we can solve numerically to find the amplitude of the oscillations, it still is of interest to have closed-form formulas even if they are only approximately correct.  These closed-form formulas or approximations provide insight into the problem and give the size of the amplitude explicitly in terms of the model parameters.  This is especially valuable in some parameter regimes where some parameters are large.  To this end, in the sequel, we provide two approximations for the amplitude using first-order and second-order Taylor expansions of the multinomial logit probability function and demonstrate in some numerical examples that these approximations are quite accurate.



Before proceeding, we define the multinomial logit function $f : \mathbb{R} \to \mathbb{R}$ such that $$f(x) = \frac{e^{-\theta x}}{e^{-\theta x} + e^{-\theta(\rho - L)}}$$ and for convenience we evaluate $f$ and some of its derivatives at values that will be used in some upcoming proofs.

\begin{eqnarray*}
f(L) &=& \frac{e^{- \theta L}}{e^{- \theta L} + e^{-\theta (\rho - L)}} =  \frac{1}{1+ e^{-\theta (\rho - 2L)}} \\
f'(L) &=& =  \frac{-2 \theta e^{-\theta (\rho - 2L)}  }{\left(1+ e^{-\theta (\rho - 2L)} \right)^2}  = - \frac{ \theta \mathrm{sech}^2\left( \theta L - \theta \frac{\rho}{2} \right) }{2} \\
f''(L) &=&  \frac{8 \theta^2 e^{-2\theta (\rho - 2L)}  }{\left(1+ e^{-\theta (\rho - 2L)} \right)^3} - \frac{4 \theta^2 e^{-\theta (\rho - 2L)}  }{\left(1+ e^{-\theta (\rho - 2L)} \right)^2}  =  \theta^2 \mathrm{tanh}\left( \theta L - \theta \frac{\rho}{2} \right) \mathrm{sech}^2\left( \theta L - \theta \frac{\rho}{2} \right)  \\
f(0) &=&  \frac{1}{1+ e^{-\theta \rho }} \\
f'(0) &=&  \frac{-2 \theta e^{-\theta \rho }  }{\left(1+ e^{-\theta \rho } \right)^2}  = - \frac{ \theta  }{2} \mathrm{sech}^2\left( - \frac{\rho \theta}{2} \right) \\
f''(0) &=&  \frac{8 \theta^2 e^{-2\theta \rho }  }{\left(1+ e^{-\theta \rho } \right)^3} - \frac{4 \theta^2 e^{-\theta \rho}  }{\left(1+ e^{-\theta \rho } \right)^2}  =  \theta^2 \mathrm{tanh}\left( - \frac{ \theta \rho}{2} \right) \mathrm{sech}^2\left( - \frac{ \theta \rho}{2} \right)  \\
f\left( \frac{\rho}{2} \right) &=&  \frac{1}{2} \\
f'\left( \frac{\rho}{2} \right) &=&  - \frac{ \theta  }{2}  \\
f''\left( \frac{\rho}{2} \right) &=&  0  \\
\end{eqnarray*} 

\begin{theorem}
We obtain the following first-order Taylor approximation for the amplitude
  \begin{eqnarray}
\text{Amplitude} &=& \frac{\rho}{2} -  \frac{ \rho - \rho \left( \frac{1 - e^{- \mu \Delta} }{1+ e^{-  \rho \theta } } \right) }{ 1 + e^{-\mu \Delta} -   \frac{ \rho \theta  }{2 } \cdot \mathrm{sech}^2\left( - \frac{\rho \theta}{2} \right)   \cdot ( 1 - e^{- \mu \Delta} )  }  .
\label{amplitude_linear_1_thm_23_equation}
\end{eqnarray}
  
  \begin{proof}
  The first step is to perform a Taylor expansion for the multinomial logit function as a function of $L$.  A first order Taylor expansion yields
  \begin{eqnarray*}
f(L) &=& \frac{e^{- \theta L}}{e^{- \theta L} + e^{-\theta (\rho - L)}} \\
&\approx& f(0) + f'(0) L \\
&=&  \frac{1}{1+ e^{-\theta \rho }}  - \left( \frac{ \theta  }{2} \mathrm{sech}^2\left( - \frac{\rho \theta}{2} \right) \right) \cdot L .
\end{eqnarray*} 

Now we substitute the Taylor approximation for the funcction $f(L)$ into the fixed-point equation and solve for the variable $L$.
This yields 
\begin{eqnarray}
 L &=&   \frac{ \rho - \rho \left( \frac{1 - e^{- \mu \Delta} }{1+ e^{-  \rho \theta } } \right) }{ 1 + e^{-\mu \Delta} -   \frac{ \rho \theta  }{2} \cdot \mathrm{sech}^2\left( - \frac{\rho \theta}{2 } \right)   \cdot ( 1 - e^{- \mu \Delta} )  } , 
\end{eqnarray}
  which implies that the amplitude is given by the following equation 
  \begin{eqnarray}
\text{Amplitude} &=& \frac{\rho}{2 } -  \frac{ \rho - \rho \left( \frac{1 - e^{- \mu \Delta} }{1+ e^{-  \rho \theta  } } \right) }{ 1 + e^{-\mu \Delta} -   \frac{ \rho \theta  }{2 } \cdot \mathrm{sech}^2\left( - \frac{\rho \theta}{2 } \right)   \cdot ( 1 - e^{- \mu \Delta} )  }  .
\end{eqnarray}
This completes the proof.  
  \end{proof}

\end{theorem}

\begin{theorem}
We obtain the following second-order Taylor approximation for the amplitude
  \begin{eqnarray}
\text{Amplitude} &=& \frac{\rho}{2 } - L^* 
\label{quadratic_approximation_thm_24}
\end{eqnarray}
where 
\begin{eqnarray}
 L^* &=& - \frac{1 + e^{-\mu \Delta} + \rho f'(0) (1 - e^{-\mu \Delta})}{\rho f''(0) (1 - e^{- \mu \Delta})} \nonumber\\
&\pm& \frac{\sqrt{(1 + e^{-\mu \Delta} + \rho f'(0)(1 - e^{-\mu \Delta}))^2 - 2  \rho^2 f'(0)(1 - e^{-\mu \Delta})(f(0)(1 - e^{-\mu \Delta}) - 1)}}{\rho f''(0)(1 - e^{-\mu \Delta})} \nonumber.
\end{eqnarray}
and we define 
\begin{eqnarray*}
f(0) &=&  \frac{1}{1+ e^{-\theta \rho }} \\
f'(0) &=&  - \frac{ \theta  }{2} \mathrm{sech}^2\left( - \frac{\rho \theta}{2} \right) \\
f''(0) &=&  \theta^2 \mathrm{tanh}\left( - \frac{ \theta \rho}{2} \right) \mathrm{sech}^2\left( - \frac{ \theta \rho}{2} \right).
\end{eqnarray*}


 \begin{proof}
  The first step is to perform a Taylor expansion for the multinomial logit function as a function of $L$.  A second order Taylor expansion yields
  \begin{eqnarray*}
f(L) &=& \frac{e^{- \theta L}}{e^{- \theta L} + e^{-\theta (\rho - L)}} \\
&\approx& f(0) + f'(0) L + \frac{f''(0)}{2} L^2 \\
&=&  \frac{1}{1+ e^{-\theta \rho }}  - \left( \frac{ \theta  }{2} \mathrm{sech}^2\left( - \frac{\rho \theta}{2} \right) \right) \cdot L +  \frac{\theta^2}{2} \mathrm{tanh}\left( - \frac{ \theta \rho}{2} \right) \mathrm{sech}^2\left( - \frac{ \theta \rho}{2} \right)  \cdot L^2.
\end{eqnarray*} 

Now we substitute the Taylor approximation for the function $f(L)$ into the fixed-point equation and solve for the variable $L$ using the quadratic formula.  This yields 
\begin{eqnarray}
 L^* &=& - \frac{1 + e^{-\mu \Delta} + \rho f'(0) (1 - e^{-\mu \Delta})}{\rho f''(0) (1 - e^{- \mu \Delta})} \nonumber\\
&\pm& \frac{\sqrt{(1 + e^{-\mu \Delta} + \rho f'(0)(1 - e^{-\mu \Delta}))^2 - 2  \rho^2 f'(0)(1 - e^{-\mu \Delta})(f(0)(1 - e^{-\mu \Delta}) - 1)}}{\rho f''(0)(1 - e^{-\mu \Delta})} \nonumber.
\end{eqnarray}
where we define 
\begin{eqnarray*}
f(0) &=&  \frac{1}{1+ e^{-\theta \rho }} \\
f'(0) &=&  - \frac{ \theta  }{2} \mathrm{sech}^2\left( - \frac{\rho \theta}{2} \right) \\
f''(0) &=&  \theta^2 \mathrm{tanh}\left( - \frac{ \rho \theta }{2} \right) \mathrm{sech}^2\left( - \frac{\rho \theta }{2} \right).
\end{eqnarray*} 
Using the amplitude formula, we complete the proof.  
  \end{proof}



\end{theorem}
 We have observed that the $+$ from the $\pm$ as the corresponding root performs better numerically than the $-$.
 
 \begin{remark}
 One might ask whether the Taylor expansion around $0$ is the right Taylor expansion to do.  We have attempted to do a Taylor expansion around the equilibrium point $\frac{\rho}{2 }$ and find that the expansion yields 
 \begin{eqnarray*}
f\left( \frac{\rho}{2 } \right) &=&  \frac{1}{2} \\
f'\left( \frac{\rho}{2 } \right) &=&  - \frac{ \theta  }{2}  \\
f''\left( \frac{\rho}{2 } \right) &=&  0  .
\end{eqnarray*} 
This seems promising since the second derivative around the equilibrium point is equal to zero.  However, what we find is that the Taylor expansion around the equilibrium only yields the equilibrium solution, which is an amplitude of zero.  Thus, the Taylor expansion around the equilibrium does not yield any information about the amplitude.  
 \end{remark}

\begin{theorem}
When the two-dimensional update system exhibits a nonzero amplitude, the first-order Taylor approximation for the amplitude is an upper bound for the actual amplitude (which is obtained by solving the nonlinear fixed-point equation). In particular, if we let $A_1$ denote the first-order Taylor approximation of the amplitude and $A$ denote the actual amplitude, then if $A > 0$ we have that $$A_1 \geq A.$$
\end{theorem}
\begin{proof}
First consider the multinomial logit function $$f(x) = \frac{1}{1 + e^{-\theta(\rho - 2x)}}$$ which is the function in the nonlinear fixed-point equation that we approximate with a first-order Taylor expansion to get our first-order approximation. The second derivative of this function is given by $$f''(x) = \theta^2 \mathrm{tanh}\left( \theta x - \theta \frac{\rho}{2} \right) \mathrm{sech}^2\left( \theta x - \theta \frac{\rho}{2} \right)  $$ and we see that $f''(x) = 0$ precisely when $x = \frac{\rho}{2}$. Additionally, $f''(x) < 0$ when $x < \frac{\rho}{2}$ and $f''(x) > 0$ when $x > \frac{\rho}{2}$.  This tells us that this function is concave when $x < \frac{\rho}{2}$ and convex when $x > \frac{\rho}{2}$. In our fixed-point equation for $L$, we have an $f(L)$ term that we ultimately take a first-order Taylor expansion of to get our first-order approximation. Recall that the nonlinear fixed-point equation is 
\begin{eqnarray}
\rho - L &=& L e^{-\mu \Delta} + \rho f(L) (1 - e^{- \mu \Delta})
\end{eqnarray}
\noindent and we can rewrite it as 
\begin{eqnarray}
C_1 + C_2 L &=& f(L)
\end{eqnarray}
\noindent where $$C_1 = \frac{1}{1 - e^{-\mu \Delta}}$$ and $$C_2 = -\frac{1 + e^{-\mu \Delta}}{\rho \left( 1 - e^{-\mu \Delta}  \right)}.$$ Note that since $\mu, \Delta, \rho > 0$, it follows that $C_2 < 0$, so $C_1 + C_2 L$ as a function of $L$ is a line with negative slope. Recall that $U + L = \rho$ and that $L \leq U$ so that $L \leq \frac{\rho}{2}$. If $L = U = \frac{\rho}{2}$, then the solution is stable and the amplitude is $0$, so we instead consider the case where $L < \frac{\rho}{2}$ which corresponds to the system having a nonzero ampltiude. In this case, we see that $f$ is concave and thus $$f(L) \leq f(0) + f'(0) L.$$ Our nonlinear fixed-point equation is $$C_1 + C_2 L = f(L)$$ and the fixed-point equation corresponding to the first-order Taylor approximation is $$C_1 + C_2 L = f(0) + f'(0) L.$$ Since $C_1 + C_2 L$ is a line with negative slope, it will intersect with a strictly larger function before it intersects with a strictly smaller function and thus if $L^*$ is the solution to the nonlinear fixed-point equation and $L_1^*$ is the solution to the first-order Taylor approximation fixed-point equation, then we have that $$L_1^* \leq L^*$$ where we note that since $f'(0) \neq C_2$,  we have that $L_1^*$ exists and is unique. The corresponding approximations for the upper value of the amplitude are $U^* := \rho - L^*$ and $U_1^* := \rho - L_1^*$, respectively, and so it follows that $$U_1^* \geq U^*$$ and so $$A_1 = \frac{U_1^* - L_1^*}{2} \geq \frac{U^* - L^*}{2} = A.$$
\end{proof}

\noindent Below, in Figures \ref{ex1}, \ref{ex2}, and \ref{ex3}, we show examples of our amplitude approximations being applied to the two-dimensional update system for various values of $\Delta$. Additionally, we collect data on the amplitude approximations for several values of $\Delta$ in Table \ref{delta_approx_table}

\begin{figure}
\centering 
\includegraphics[scale=.3]{./Figures/FP_ex2_hq.jpg}~\includegraphics[scale=.3]{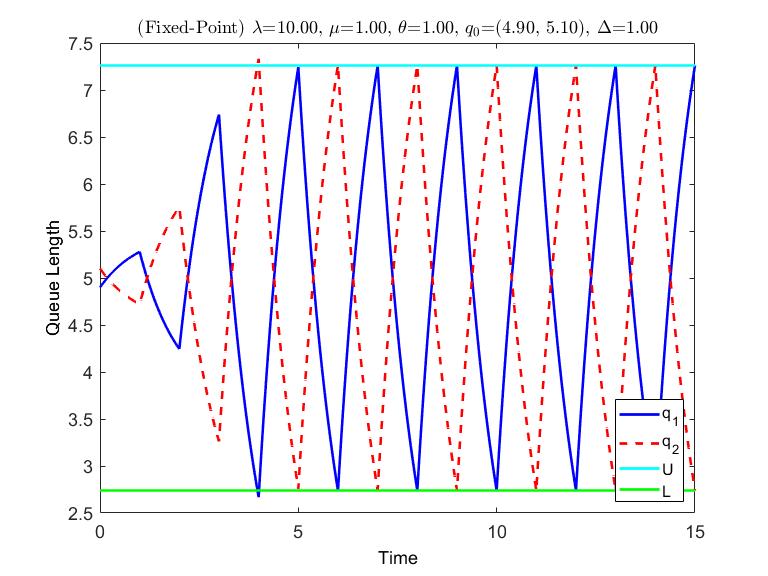}
\caption{The queue lengths plotted against time where the horizontal lines represent our approximations of the amplitude of the oscillations by numerically solving the nonlinear fixed-point equation for $L$. These plots are for $\lambda = 10, \mu = 1, \theta=1$, and for $t \in [- \Delta, 0]$ we have $q_1 = .49$ and $q_2 = .51.$ In the left plot $\Delta=.7$ and in the right plot $\Delta = 1$.}
\label{ex1}
\end{figure}

\begin{figure}
\centering 
\includegraphics[scale=.3]{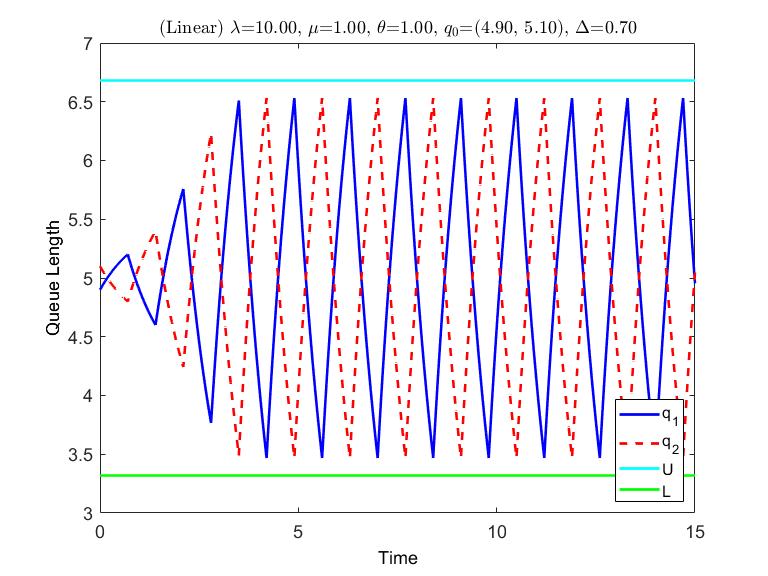}~\includegraphics[scale=.3]{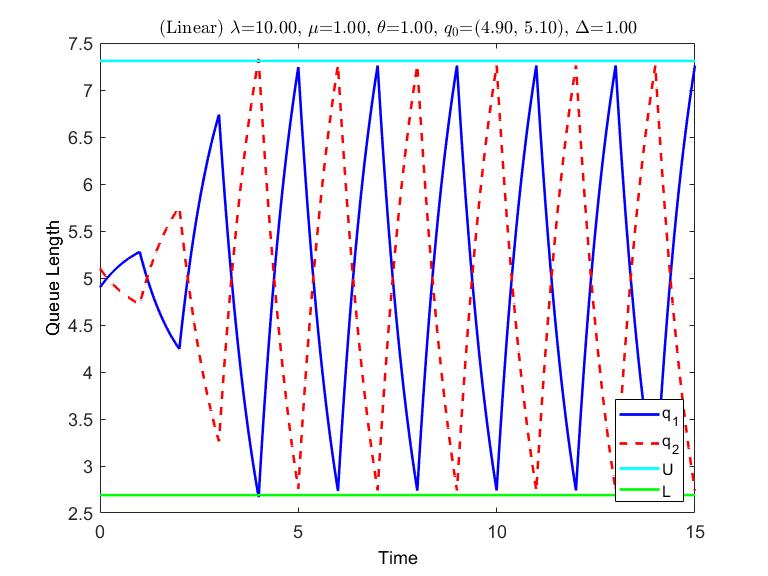}
\caption{The queue lengths plotted against time where the horizontal lines represent our approximations of the amplitude of the oscillations by numerically solving the first-order Taylor approximation fixed-point equation for $L$. These plots are for $\lambda = 10, \mu = 1, \theta=1$, and for $t \in [- \Delta, 0]$ we have $q_1 = .49$ and $q_2 = .51.$ In the left plot $\Delta=.7$ and in the right plot $\Delta = 1$.}
\label{ex2}
\end{figure}

\begin{figure}
\centering 
\includegraphics[scale=.3]{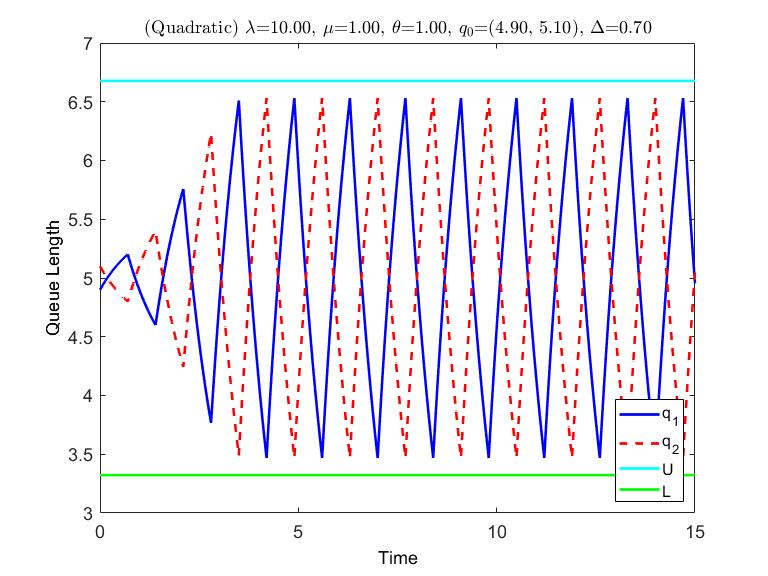}~\includegraphics[scale=.3]{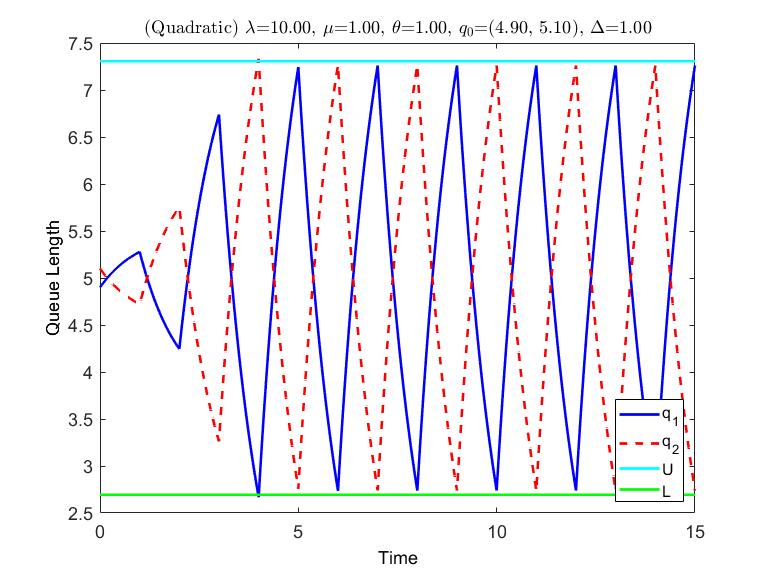}
\caption{The queue lengths plotted against time where the horizontal lines represent our approximations of the amplitude of the oscillations by numerically solving the second-order Taylor approximation fixed-point equation for $L$. These plots are for $\lambda = 10, \mu = 1, \theta=1$, and for $t \in [- \Delta, 0]$ we have $q_1 = .49$ and $q_2 = .51.$ In the left plot $\Delta=.7$ and in the right plot $\Delta = 1$.}
\label{ex3}
\end{figure}

%
%

\begin{table}
\begin{center}
\begin{tabular}{ |c | c | c | c| }
\hline
 $\Delta$ & Fixed-Point & Linear & Quadratic \\ 
\hline
0.60 & 1.2253 & 1.4555 & 1.4522\\ 
\hline
0.80 & 1.7984 & 1.8985 & 1.8952\\ 
\hline
1.00 & 2.2609 & 2.3092 & 2.3062\\ 
\hline
1.20 & 2.6590 & 2.6839 & 2.6813\\ 
\hline
1.40 & 3.0071 & 3.0205 & 3.0183\\ 
\hline
1.60 & 3.3114 & 3.3189 & 3.3172\\ 
\hline
1.80 & 3.5759 & 3.5802 & 3.5789\\ 
\hline
2.00 & 3.8042 & 3.8068 & 3.8058\\ 
\hline
2.20 & 3.9998 & 4.0014 & 4.0007\\ 
\hline
2.40 & 4.1663 & 4.1673 & 4.1667\\ 
\hline
2.60 & 4.3071 & 4.3077 & 4.3073\\ 
\hline
2.80 & 4.4255 & 4.4259 & 4.4256\\ 
\hline
3.00 & 4.5247 & 4.5249 & 4.5248\\ 
\hline
\end{tabular}
\end{center}
\caption{The values of the amplitude of the oscillations in the two-delay update system for various values of $\Delta$ according to three different approximations. The Fixed-Point column corresponds to numerically solving the nonlinear fixed-point equation for $L$ and then computing $\frac{\lambda}{2 \mu} - L$ to obtain the amplitude. The Linear column uses a first-order Taylor expansion of the multinomial logit function in the fixed-point equation before solving for $L$ and the Quadratic column does the same except with a second-order Taylor expansion. Other parameters used were $\lambda=10$, $\mu=1$, and $\theta=1$ for all cases.}
\label{delta_approx_table}
\end{table}

\section{The Multi-dimensional Case}
\label{multidimensional_case_section}

In the N-dimensional case, we have the following system of functional delay differential equations
\begin{eqnarray} \label{fdesN}
\overset{\bullet}{q}_i(t) = \lambda \cdot \frac{\exp \left( - \theta \cdot q_i(\Phi(t, \Delta)) \right) }{\sum_{j=1}^{N} \exp\left( -\theta \cdot q_j( \Phi(t, \Delta) \right)} - \mu q_i(t), \hspace{5mm} i = 1, 2, ... , N.
\end{eqnarray}

\noindent where $$\Phi(t, \Delta) := \bigg\lfloor \frac{t}{\Delta} \bigg\rfloor \Delta.$$
Like in the two-dimensional case, the amplitude can be obtained by solving a system of 2N nonlinear equations i.e.

 \begin{eqnarray}  \label{UequationN}
U_i &=&  L_i e^{- \mu \Delta} + \rho \frac{e^{-\theta \cdot L_i}}{ \sum^{N}_{j=1} e^{-\theta \cdot L_j} } (1 - e^{- \mu \Delta}) \\
L_i &=&  U_i e^{- \mu \Delta} + \rho \frac{e^{-\theta \cdot U_i}}{ \sum^{N}_{j=1} e^{-\theta \cdot U_j} } (1 - e^{- \mu \Delta}). 
\end{eqnarray}

Unfortunately, just like in the two-dimensional case, we cannot find an explicit closed-form solution to the system of equations.  Moreover, what is also true is that the upper and lower intrepretations of the $U_i$ and $L_i$ no longer apply in the case where $N > 2$.  This can be observed in the example given in Figure \ref{3D_example}.  

\begin{figure}
\includegraphics[scale=.4]{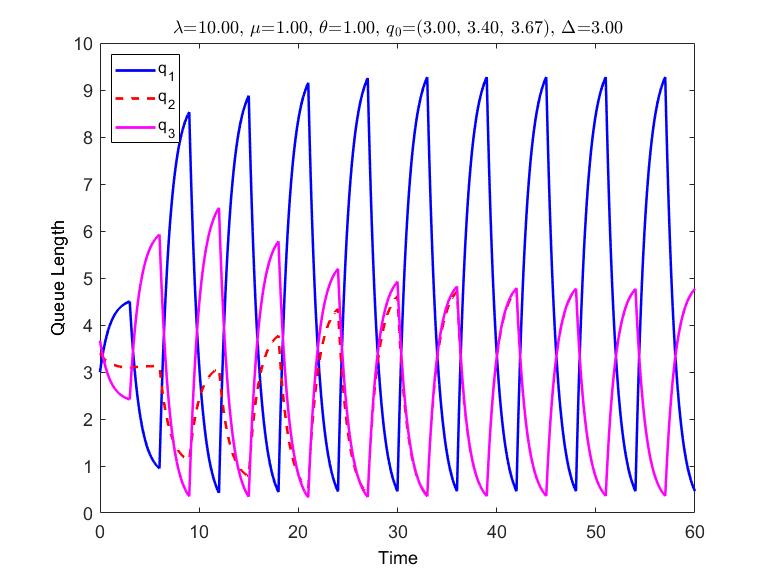}
\caption{A plot of the queue lengths for the update system with $N=3$ against time. We see that the first queue, $q_1$, approaches a different limiting amplitude than the other two queues do.}
\label{3D_example}
\end{figure}

\subsection{The Even Number of Queues Case}

In this case, we can reduce the problem back to the two dimensional case.  We observe from our numerical examples that the even dimensional case always reduces to a setting where the amplitudes are all the same size for all queues. In particular there are two sets of queues, each containing $\frac{N}{2}$ identically-behaving (approximately, for large time) queues, and both sets of queues have the same amplitude but they are out of phase with each other. In this case, we can map this to the two dimensional case where the arrival rate is suppressed.  This implies that 

\begin{eqnarray}
\frac{\lambda^*}{\mu} - L = L e^{-\mu \Delta} + \frac{\lambda^*}{\mu} \frac{e^{-\theta \cdot L}}{e^{-\theta \cdot L} + e^{- \theta \cdot (\frac{\lambda^*}{\mu} - L)}}(1 - e^{- \mu \Delta}). \label{subequation}
\end{eqnarray}
where $\lambda^* = 2 \lambda/N$. This factor of $\frac{2}{N}$ is applied to the original arrival rate $\lambda$ because the arrival rate is shared equally amongst all $\frac{N}{2}$ queues in each set of queues.

Moreover, the Taylor expansion approximations can also be modified by replacing $\lambda$ with $\lambda^*$ to compute the amplitude.  In Tables \ref{delta_approx_table_N_4}-\ref{delta_approx_table_N_8} we compare the amplitude obtained by numerically solving the fixed-point equation \ref{subequation} (after scaling the arrival rate by $\frac{2}{N}$, as mentioned above) with the corresponding linear and quadratic approximations introduced in Equation \ref{amplitude_linear_1_thm_23_equation} and Equation \ref{quadratic_approximation_thm_24}, respectively. We compare the amplitude with the linear and quaratic approximations for several values of $\Delta$ and we do so for $N=4, N=6$, and $N=8$. In each of these tables, the Fixed-Point column corresponds to the amplitude obtained by numerically solving the nonlinear fixed-point equation for $L$ and then computing $\frac{\lambda^*}{2 \mu} - L$ to obtain the amplitude. The Linear column uses a first-order Taylor expansion of the multinomial logit function in the fixed-point equation before solving for $L$ and the Quadratic column does the same except with a second-order Taylor expansion. 

From the tables, we see that the linear approximation appears to get more accurate as $\Delta$ increases and less accurate as $N$ increases. While the quadratic approximation is more accurate than the linear approximation in some cases, there are many cases for which the linear approximation is still more accurate. Overall, the linear approximation appears to be more reliable than the quadratic approximation. 


\begin{table}[H]
\begin{center}
\begin{tabular}{ |c | c | c | c| }
\hline
 $\Delta$ & Fixed-Point & Linear & Quadratic \\ 
\hline
0.90 & 0.4075 & 0.9973 & 0.9263\\ 
\hline
1.10 & 0.8904 & 1.1898 & 1.1251\\ 
\hline
1.30 & 1.1847 & 1.3658 & 1.3102\\ 
\hline
1.50 & 1.4088 & 1.5243 & 1.4786\\ 
\hline
1.70 & 1.5896 & 1.6651 & 1.6288\\ 
\hline
1.90 & 1.7386 & 1.7888 & 1.7606\\ 
\hline
2.10 & 1.8625 & 1.8961 & 1.8748\\ 
\hline
2.30 & 1.9657 & 1.9885 & 1.9725\\ 
\hline
2.50 & 2.0518 & 2.0673 & 2.0555\\ 
\hline
2.70 & 2.1235 & 2.1340 & 2.1254\\ 
\hline
2.90 & 2.1831 & 2.1903 & 2.1840\\ 
\hline
3.10 & 2.2325 & 2.2375 & 2.2329\\ 
\hline
3.30 & 2.2734 & 2.2769 & 2.2736\\ 
\hline
\end{tabular}
\quad
\begin{tabular}{ |c | c | c | c| }
\hline
 $\Delta$ & Fixed-Point & Linear & Quadratic \\ 
\hline
1.40 & 0.1479 & 0.8050 & 0.6458\\ 
\hline
1.60 & 0.5719 & 0.9007 & 0.7603\\ 
\hline
1.80 & 0.7775 & 0.9875 & 0.8679\\ 
\hline
2.00 & 0.9235 & 1.0650 & 0.9657\\ 
\hline
2.20 & 1.0354 & 1.1333 & 1.0521\\ 
\hline
2.40 & 1.1240 & 1.1930 & 1.1270\\ 
\hline
2.60 & 1.1952 & 1.2445 & 1.1910\\ 
\hline
2.80 & 1.2529 & 1.2885 & 1.2451\\ 
\hline
3.00 & 1.3000 & 1.3260 & 1.2904\\ 
\hline
3.20 & 1.3384 & 1.3576 & 1.3281\\ 
\hline
3.40 & 1.3698 & 1.3842 & 1.3594\\ 
\hline
3.60 & 1.3956 & 1.4065 & 1.3853\\ 
\hline
3.80 & 1.4167 & 1.4250 & 1.4067\\ 
\hline
\end{tabular}
\end{center}
\caption{$N = 4$ (Left) and  $N = 6$ (Right) \\ $\lambda=10$, $\mu=1$, and $\theta=1$}
\label{delta_approx_table_N_4}
\end{table}

\begin{table}[H]
\begin{center}
\begin{tabular}{ |c | c | c | c| }
\hline
 $\Delta$ & Fixed-Point & Linear & Quadratic \\ 
\hline
2.20 & 0.0433 & 0.5980 & 0.0555\\ 
\hline
2.40 & 0.3568 & 0.6393 & 0.1327\\ 
\hline
2.60 & 0.4848 & 0.6760 & 0.1934\\ 
\hline
2.80 & 0.5718 & 0.7083 & 0.2416\\ 
\hline
3.00 & 0.6363 & 0.7363 & 0.2799\\ 
\hline
3.20 & 0.6858 & 0.7605 & 0.3107\\ 
\hline
3.40 & 0.7248 & 0.7811 & 0.3355\\ 
\hline
3.60 & 0.7557 & 0.7985 & 0.3555\\ 
\hline
3.80 & 0.7806 & 0.8132 & 0.4367\\ 
\hline
4.00 & 0.8006 & 0.8256 & 0.5038\\ 
\hline
4.20 & 0.8168 & 0.8359 & 0.5454\\ 
\hline
4.40 & 0.8299 & 0.8445 & 0.5757\\ 
\hline
4.60 & 0.8406 & 0.8516 & 0.5989\\ 
\hline
\end{tabular}
\end{center}
\caption{$N = 8, \lambda=10$, $\mu=1$, and $\theta=1$}
\label{delta_approx_table_N_8}
\end{table}

Next we consider various plots of queue length against time in Figures \ref{EVEN_PLOTS_1}-\ref{EVEN_PLOTS_9} for $N=4, N=6,$ and $N= 8$. In Figure \ref{EVEN_PLOTS_1} we plot queue length against time for all four of the queues in the $N=4$ update system. We see that there are two sets of two queues where the queues in different sets all have the same amplitude and only differ by a phase shift. In Figure \ref{EVEN_PLOTS_2} we plot amplitude bars obtained from solving the fixed-point equation and in Figure \ref{EVEN_PLOTS_3} we also plot amplitude bars obtained from the linear amplitude approximation.

In Figure \ref{EVEN_PLOTS_4} we plot queue length against time for all four of the queues in the $N=6$ update system. We see that there are two sets of three queues where the queues in different sets all have the same amplitude and only differ by a phase shift. In Figure \ref{EVEN_PLOTS_5} we plot amplitude bars obtained from solving the fixed-point equation and in Figure \ref{EVEN_PLOTS_6} we also plot amplitude bars obtained from the linear amplitude approximation.

 In Figure \ref{EVEN_PLOTS_7} we plot queue length against time for all four of the queues in the $N=8$ update system. We see that there are two sets of four queues where the queues in different sets all have the same amplitude and only differ by a phase shift. In Figure \ref{EVEN_PLOTS_8} we plot amplitude bars obtained from solving the fixed-point equation and in Figure \ref{EVEN_PLOTS_9} we also plot amplitude bars obtained from the linear amplitude approximation.

Overall, we see that the plots reiterate the conclusion made from the tables that the linear amplitude approximation tends to get worse as $N$ increases as all of the plots have the same value of $\Delta = 3$ and the amplitude approximation is considerably worse in the $N = 8$ case than in the other cases.



\begin{figure}
\includegraphics[scale=.65 ]{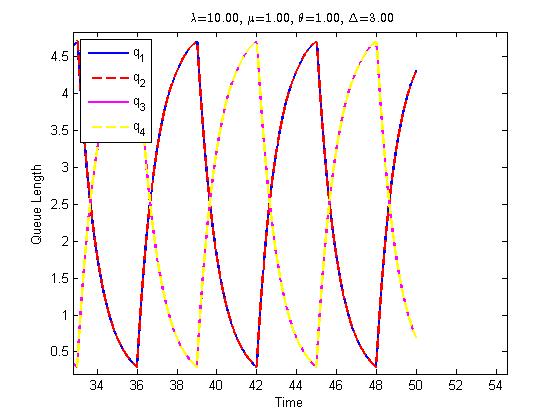}
\caption{$N=4$}
\label{EVEN_PLOTS_1}
\end{figure}

\begin{figure}
\includegraphics[scale=.65 ]{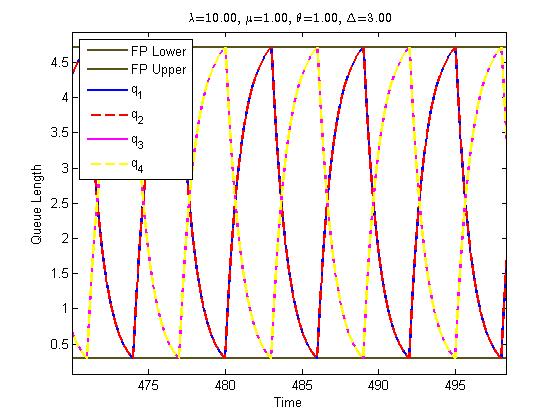}
\caption{$N=4$}
\label{EVEN_PLOTS_2}
\end{figure}

\begin{figure}
\includegraphics[scale=.65 ]{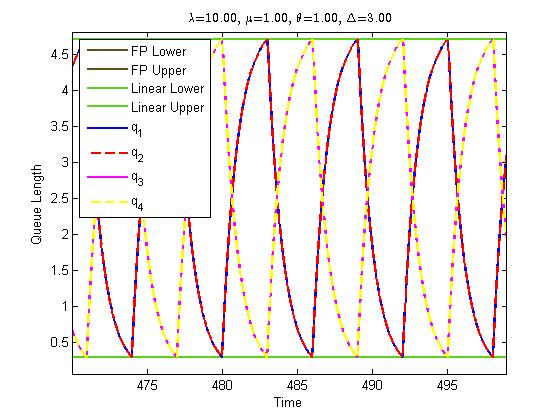}
\caption{$N=4$}
\label{EVEN_PLOTS_3}
\end{figure}


\begin{figure}
\includegraphics[scale=.65 ]{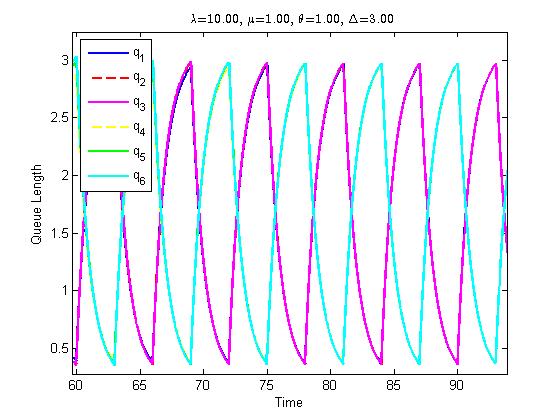}
\caption{$N=6$}
\label{EVEN_PLOTS_4}
\end{figure}

\begin{figure}
\includegraphics[scale=.65 ]{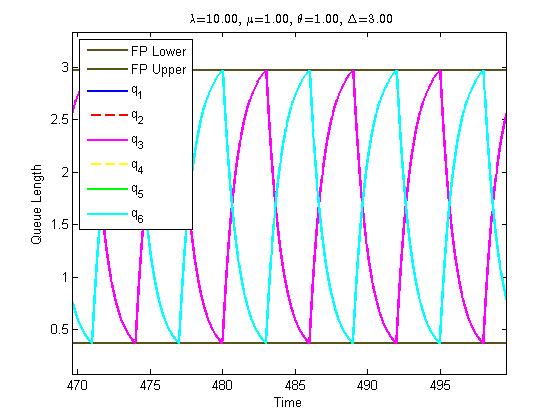}
\caption{$N=6$}
\label{EVEN_PLOTS_5}
\end{figure}

\begin{figure}
\includegraphics[scale=.65 ]{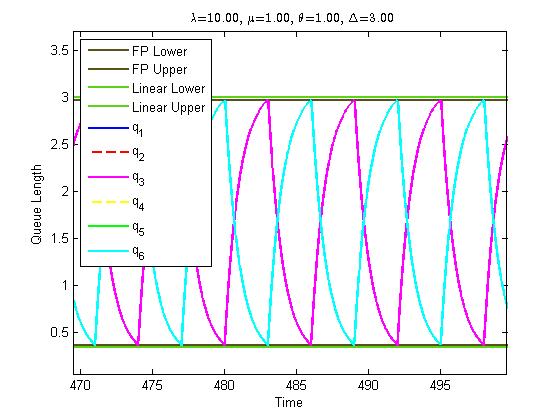}
\caption{$N=6$}
\label{EVEN_PLOTS_6}
\end{figure}


\begin{figure}
\includegraphics[scale=.65 ]{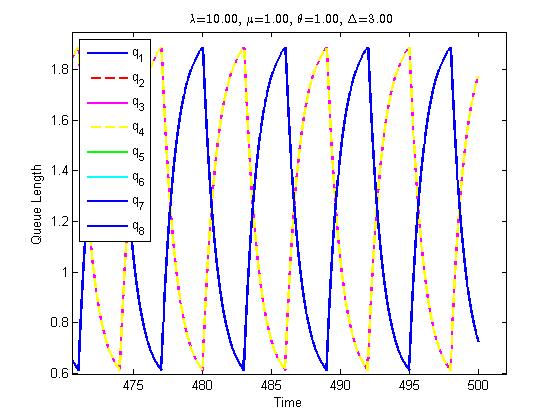}
\caption{$N=8$}
\label{EVEN_PLOTS_7}
\end{figure}

\begin{figure}
\includegraphics[scale=.65 ]{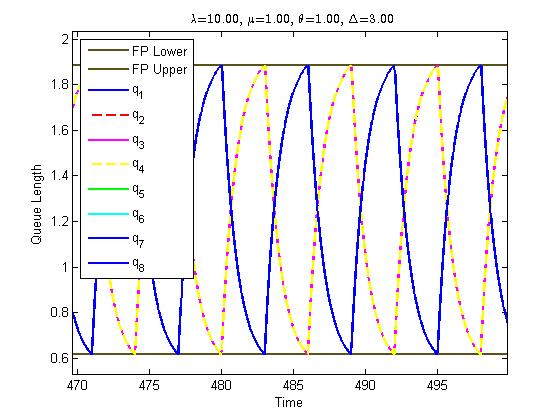}
\caption{$N=8$}
\label{EVEN_PLOTS_8}
\end{figure}

\begin{figure}
\includegraphics[scale=.65 ]{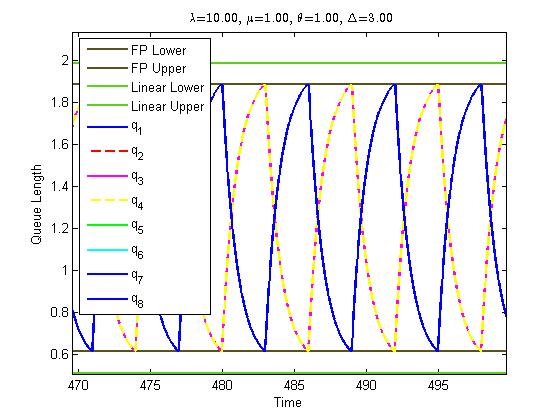}
\caption{$N=8$}
\label{EVEN_PLOTS_9}
\end{figure}

\subsection{The Odd Number of Queues Case}



When $N$ is odd, numerics show that $\frac{N-1}{2}$ of the queues are in phase with each other and have the same amplitudes whereas the other $\frac{N+1}{2}$ queues are in phase with each other and approach another amplitude. With this insight, we can reduce the system of $2N$ equations to the following system of four equations. 

\begin{eqnarray}
U_1 &=& L_1 e^{-\mu \Delta} + \rho \frac{e^{-\theta L_1}}{ \left( \frac{N+1}{2}  \right) e^{-\theta L_1} + \left(  \frac{N-1}{2} \right) e^{-\theta U_2} } (1 - e^{-\mu \Delta}) \label{odd_system_U1}\\
L_1 &=& U_1 e^{-\mu \Delta} + \rho \frac{e^{-\theta U_1}}{ \left( \frac{N+1}{2}  \right) e^{-\theta U_1} + \left(  \frac{N-1}{2} \right) e^{-\theta L_2} } (1 - e^{-\mu \Delta}) \label{odd_system_L1}\\
U_2 &=& L_2 e^{-\mu \Delta} + \rho \frac{e^{-\theta L_2}}{ \left( \frac{N-1}{2}  \right) e^{-\theta L_2} + \left(  \frac{N+1}{2} \right) e^{-\theta U_1} } (1 - e^{-\mu \Delta}) \label{odd_system_U2}\\
L_2 &=& U_2 e^{-\mu \Delta} + \rho \frac{e^{-\theta U_2}}{ \left( \frac{N-1}{2}  \right) e^{-\theta U_2} + \left(  \frac{N+1}{2} \right) e^{-\theta L_1} } (1 - e^{-\mu \Delta}) \label{odd_system_L2}
\end{eqnarray}

\noindent This is similar to the two-dimensional case except the two queues are really two sets of queues, both of which have different amplitudes, one of which contains $\frac{N+1}{2}$ queues with the same amplitude and the other contains $\frac{N-1}{2}$ queues with the same amplitude. We account for these different amplitudes by scaling the arrival rate corresponding to one of the sets of queues by $\frac{2}{N + 1}$ and the other by $\frac{2}{N - 1}$ because the true arrival rate is split evenly amongst each queue in each set of queues. In equations \ref{odd_system_U1}-\ref{odd_system_L2}, the variables $L_1$ and $U_1$ can be interpreted as the lower and upper values of the amplitude corresponding to the set of queues containing $\frac{N+1}{2}$ queues whereas $L_2$ and $U_2$ are the analogous values corresponding to the set of queues containing $\frac{N-1}{2}$ queues. This is illustrated in Figure \ref{odd_queues_picture} and we can also see from this figure that one set of queues attains the value $L_1$ at the same time that the other set of queues attains $U_2$ and similarly $U_1$ and $L_2$ are attained by their respective sets of queues at the same time. This motivates the following substituions.
\begin{eqnarray}
\left( \frac{N+1}{2} \right)U_1 +  \left(\frac{N - 1}{2} \right) L_2 &=& \rho\label{odd_sub1}\\
\left( \frac{N+1}{2} \right)L_1 +  \left(\frac{N - 1}{2} \right) U_2 &=& \rho \label{odd_sub2}
\end{eqnarray}

\noindent Using these substitutions, we can further reduce this system of four equations to be the following system of two equations.

\begin{figure}
\includegraphics[scale=.65 ]{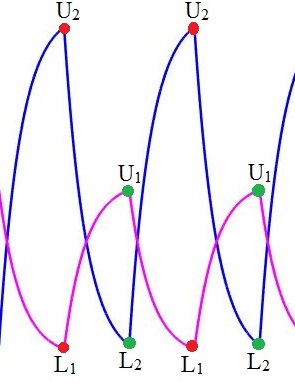}
\caption{Above is a plot of queue length (vertical axis) against time (horizontal axis). The magenta curve corresponds to the set of $\frac{N+1}{2}$ queues whose lengths are all overlapping with each other and all have the same amplitude of $\frac{U_1 - L_1}{2}$. The blue curve corresponds to the set of $\frac{N-1}{2}$ queues whose lengths are all overlapping with each other and all have the same amplitude of $\frac{U_2 - L_2}{2}$.}
\label{odd_queues_picture}
\end{figure}

\begin{eqnarray}
U_1 &=& L_1 e^{-\mu \Delta} + \rho \frac{e^{- \theta L_1}}{\left( \frac{N+1}{2} \right) e^{-\theta L_1} + \left(\frac{N - 1}{2} \right) e^{- \frac{2 \rho \theta}{(N-1) }} e^{\theta \left( \frac{N+1}{N-1} \right) L_1}}(1 - e^{-\mu \Delta}) \label{odd_2_eqns_eq1}\\
L_1 &=& U_1 e^{-\mu \Delta} + \rho \frac{e^{- \theta U_1}}{\left( \frac{N+1}{2} \right) e^{-\theta U_1} + \left(\frac{N - 1}{2} \right) e^{- \frac{2 \rho \theta}{(N-1) }} e^{\theta \left( \frac{N+1}{N-1} \right) U_1}}(1 - e^{-\mu \Delta}) \label{odd_2_eqns_eq2}
\end{eqnarray}

\noindent This system of two nonlinear Equations \ref{odd_2_eqns_eq1}-\ref{odd_2_eqns_eq2} can be solved numerically to give us $L_1$ and $U_1$ which can be used to obtain $L_2$ and $U_2$ by using Equations \ref{odd_sub1}-\ref{odd_sub2} and from these we can obtain the amplitudes for each set of queues, which will be $\frac{U_1 - L_1}{2}$ for the set of queues containing $\frac{N+1}{2}$ queues and $\frac{U_2 - L_2}{2}$ for the set of queues containing $\frac{N-1}{2}$ queues.

\begin{remark}
From Equations \ref{odd_sub1}-\ref{odd_sub2}, we see that $$ \left( \frac{N+1}{N-1} \right) \frac{U_1 - L_1}{2} =  \frac{U_2 - L_2}{2}$$ and thus the amplitude corresponding to the set of $N-1$ queues will differ from the amplitude corresponding to the remaining $N+1$ queues by a factor of $\frac{N+1}{N-1}$. This tells us that in the limit as $N \to \infty$ (and implicitly as $\Delta \to \infty$ because the critical delay increases as $N$ increases and it only makes sense to discuss amplitudes when the system is unstable), the two amplitudes will approach the same value. Additionally, the amplitudes approach $0$ and the linear approximations of the amplitudes that we found also approach $0$ in each case. This is unsurprising as the arrival rate is held fixed and needs to be distributed among increasingly many queues as $N \to \infty$.
\end{remark}

Now that we know that we can numerically solve a system of two nonlinear equations to find the amplitude of each queue, it is reasonable to search for closed-form approximations of the amplitudes. We present such approximations in the following theorem.


\begin{theorem}
When $N$ is an odd integer such that $N \geq 3$, we can obtain linear approximations of the amplitudes which are $$A_1^{(1)} := \frac{(a - 1) L_1^{(1)} + b \cdot g\left( L_1^{(1)} \right)}{2}$$ (which corresponds to the set of queues containing $\frac{N+1}{2}$ queues) and $$A_2^{(1)} := \left( \frac{N+1}{N-1} \right) A_1^{(1)}$$ (which corresponds to the remaining $\frac{N-1}{2}$ queues) where $$L_1^{(1)} =  \frac{ab \cdot g(0) + b \cdot g(b \cdot g(0))}{1 - a^2 - ab \cdot g'(0) - g'(b \cdot g(0))(ab + b^2 \cdot g'(0))}$$ where we define the function $g : \mathbb{R} \to \mathbb{R}$ such that $$g(x) := \frac{e^{-\theta x}}{\left( \frac{N+1}{2} \right) e^{-\theta x} + \left( \frac{N - 1}{2} \right) e^{-\frac{2 \rho \theta}{N-1 }} e^{\theta \left( \frac{N+1}{N-1} \right)x} }$$ and we let 
\begin{align*}
a &:= e^{- \mu \Delta}\\
b &:= \rho (1 - e^{-\mu \Delta}).
\end{align*}

\end{theorem}

\begin{proof}

\noindent With the stated definitions, we can concisely rewrite Equations \ref{odd_2_eqns_eq1}-\ref{odd_2_eqns_eq2} as follows.
\begin{eqnarray}
U_1 &=& a L_1 + b \cdot g(L_1) \label{concise_odd_eq1}\\
L_1 &=& a U_1 + b \cdot g(U_1) \label{concise_odd_eq2}
\end{eqnarray}

\noindent To get our approximation for the amplitude, we will take the expression for $U_1$ given by Equation \ref{concise_odd_eq1} and substitute it into Equation \ref{concise_odd_eq2} and then linearize the resulting nonlinear function about $L_1 = 0$. Following this approach, we have that Equation \ref{concise_odd_eq2} becomes
\begin{align}
L_1 &= a (a L_1 + b \cdot g(L_1)) + b \cdot g(aL_1 + b \cdot g(L_1))\\
&= a^2 L_1 + ab \cdot g(L_1) + b \cdot g(aL_1 + b \cdot g(L_1))\\
&\approx a^2 L_1 + ab [g(0) + g'(0) L_1] + b [g(b \cdot g(0)) + g'(b \cdot g(0)) (a + b \cdot g'(0)) L_1 ]
\end{align}

\noindent where 
\begin{align}
g(0) &= \frac{2}{(N+1) + (N-1) e^{- \frac{2 \rho \theta}{N-1}}}\\
g'(0) &= \frac{-2 \theta (N+1) e^{- \frac{2 \rho \theta}{N-1}}}{\left( (N+1) + (N-1) e^{- \frac{2 \rho \theta}{N - 1}} \right)^2}
\end{align}

\noindent and this gives us that 
\begin{align}
L_1 \approx L_1^{(1)} := \frac{ab \cdot g(0) + b \cdot g(b \cdot g(0))}{1 - a^2 - ab \cdot g'(0) - g'(b \cdot g(0))(ab + b^2 \cdot g'(0))}
\end{align}
\noindent and that 
\begin{align}
U_1 \approx U_1^{(1)} := a L_1^{(1)} + b \cdot g \left(L_1^{(1)} \right)
\end{align}
\noindent so that our amplitude approximation is 
\begin{align}
A_1^{(1)} := \frac{U_1^{(1)} - L_1^{(1)}}{2} \label{odd_amp1_approx} = \frac{(a - 1) L_1^{(1)} + b \cdot g\left( L_1^{(1)} \right)}{2}. 
\end{align}

\noindent Similarly, we define the approximations $L_2^{(1)}$ and $U_2^{(1)}$ according to the equations 
\begin{eqnarray}
\left( \frac{N+1}{2} \right)U_1^{(1)} +  \left(\frac{N - 1}{2} \right) L_2^{(1)} &=& \rho\\
\left( \frac{N+1}{2} \right)L_1^{(1)} +  \left(\frac{N - 1}{2} \right) U_2^{(1)} &=& \rho 
\end{eqnarray}

\noindent so that
\begin{align}
L_2^{(1)} &:= \frac{2 \rho}{N - 1} - \left( \frac{N+1}{N-1} \right) U_1^{(1)}\\
&:= \frac{2 \rho}{N - 1} - \left( \frac{N+1}{N-1} \right) \left( a L_1^{(1)} + b \cdot g \left( L_1^{(1)} \right) \right)
\end{align}
\noindent and 
\begin{align}
U_2^{(1)} := \frac{2 \rho}{N - 1} - \left( \frac{N+1}{N-1} \right) L_1^{(1)}
\end{align}

\noindent and we let the approximation to the other amplitude be 
\begin{align}
A_2^{(1)} &:= \frac{U_2^{(1)} - L_2^{(1)}}{2} \label{odd_amp2_approx} \\
&= \left( \frac{N+1}{N-1} \right) \left( \frac{(a - 1) L_1^{(1)} + b \cdot g\left( L_1^{(1)} \right)}{2} \right)\\
&= \left( \frac{N+1}{N-1} \right) A_1^{(1)}
\end{align}

\end{proof}

We explore the accuracy of these amplitude approximations below in Tables \ref{delta_approx_table_N_3}-\ref{delta_approx_table_N_9} where we consider $N=3, N=5, N=7,$ and $N=9$ (respectively) and compare solving the nonlinear system \ref{odd_2_eqns_eq1}-\ref{odd_2_eqns_eq2} to obtain the amplitude against the linear approximations obtained in \ref{odd_amp1_approx} and \ref{odd_amp2_approx}. In each of these tables, we show values corresponding to each of the two amplitudes present in the system, whose values differ by a factor of $\frac{N+1}{N-1}$, for various values of $\Delta$. In the Nonlinear (1) and Nonlinear (2) columns, we give the value of each of the amplitudes obtained by numerically solving the nonlinear system \ref{odd_2_eqns_eq1}-\ref{odd_2_eqns_eq2} and then computing $\frac{U_1 - L_1}{2}$ and $\frac{U_2 - L_2}{2}$ to obtain the amplitudes. In the Linear (1) and Linear (2) columns, we give the approximate amplitude values obtained by using the first-order Taylor approximation that we came up with in Equations \ref{odd_amp1_approx} and \ref{odd_amp2_approx}. We can observe that the linear amplitude approximation tends to get more accurate as $\Delta$ increases and less accurate as $N$ increases. Additionally, the linear approximation appears to be an upper bound for the actual value of the amplitude.



\begin{table}[H]
\begin{center}
\begin{tabular}{ |c | c | c | c | c| }
\hline
 $\Delta$ & Nonlinear (1) & Linear (1) & Nonlinear (2) & Linear (2) \\ 
\hline
1.50 & 1.3793 & 1.4267 & 2.7587 & 2.8534\\ 
\hline
1.70 & 1.5553 & 1.5805 & 3.1107 & 3.1609\\ 
\hline
1.90 & 1.7081 & 1.7192 & 3.4162 & 3.4385\\ 
\hline
2.10 & 1.8381 & 1.8427 & 3.6762 & 3.6854\\ 
\hline
2.30 & 1.9473 & 1.9491 & 3.8946 & 3.8983\\ 
\hline
2.50 & 2.0383 & 2.0391 & 4.0766 & 4.0781\\ 
\hline
2.70 & 2.1138 & 2.1141 & 4.2276 & 4.2282\\ 
\hline
2.90 & 2.1763 & 2.1764 & 4.3525 & 4.3528\\ 
\hline
3.10 & 2.2279 & 2.2279 & 4.4557 & 4.4559\\ 
\hline
3.30 & 2.2704 & 2.2705 & 4.5409 & 4.5409\\ 
\hline
3.50 & 2.3055 & 2.3055 & 4.6110 & 4.6110\\ 
\hline
3.70 & 2.3344 & 2.3344 & 4.6687 & 4.6687\\ 
\hline
3.90 & 2.3581 & 2.3581 & 4.7162 & 4.7162\\ 
\hline
\end{tabular}
\end{center}
\caption{$N = 3, \lambda=10$, $\mu=1$, and $\theta=1$}
\label{delta_approx_table_N_3}
\end{table}

\begin{table}[H]
\begin{center}
\begin{tabular}{ |c | c | c | c | c| }
\hline
 $\Delta$ & Nonlinear (1) & Linear (1) & Nonlinear (2) & Linear (2) \\ 
\hline
1.90 & 1.0046 & 1.0304 & 1.5068 & 1.5456\\ 
\hline
2.10 & 1.1030 & 1.1184 & 1.6545 & 1.6777\\ 
\hline
2.30 & 1.1841 & 1.1938 & 1.7761 & 1.7906\\ 
\hline
2.50 & 1.2511 & 1.2575 & 1.8766 & 1.8862\\ 
\hline
2.70 & 1.3065 & 1.3109 & 1.9597 & 1.9664\\ 
\hline
2.90 & 1.3523 & 1.3555 & 2.0285 & 2.0333\\ 
\hline
3.10 & 1.3902 & 1.3926 & 2.0853 & 2.0889\\ 
\hline
3.30 & 1.4214 & 1.4233 & 2.1322 & 2.1349\\ 
\hline
3.50 & 1.4472 & 1.4487 & 2.1708 & 2.1730\\ 
\hline
3.70 & 1.4684 & 1.4696 & 2.2026 & 2.2044\\ 
\hline
3.90 & 1.4858 & 1.4869 & 2.2287 & 2.2303\\ 
\hline
4.10 & 1.5002 & 1.5010 & 2.2503 & 2.2516\\ 
\hline
4.30 & 1.5119 & 1.5127 & 2.2679 & 2.2691\\ 
\hline
\end{tabular}
\end{center}
\caption{$N = 5, \lambda=10$, $\mu=1$, and $\theta=1$}
\label{delta_approx_table_N_5}
\end{table}

\begin{table}[H]
\begin{center}
\begin{tabular}{ |c | c | c | c | c| }
\hline
 $\Delta$ & Nonlinear (1) & Linear (1) & Nonlinear (2) & Linear (2) \\ 
\hline
2.60 & 0.7299 & 0.7887 & 0.9732 & 1.0517\\ 
\hline
2.80 & 0.7847 & 0.8303 & 1.0463 & 1.1071\\ 
\hline
3.00 & 0.8285 & 0.8651 & 1.1046 & 1.1535\\ 
\hline
3.20 & 0.8638 & 0.8941 & 1.1517 & 1.1921\\ 
\hline
3.40 & 0.8923 & 0.9181 & 1.1898 & 1.2241\\ 
\hline
3.60 & 0.9155 & 0.9380 & 1.2207 & 1.2507\\ 
\hline
3.80 & 0.9345 & 0.9544 & 1.2459 & 1.2726\\ 
\hline
4.00 & 0.9499 & 0.9680 & 1.2665 & 1.2906\\ 
\hline
4.20 & 0.9625 & 0.9791 & 1.2833 & 1.3055\\ 
\hline
4.40 & 0.9728 & 0.9883 & 1.2970 & 1.3177\\ 
\hline
4.60 & 0.9812 & 0.9958 & 1.3083 & 1.3277\\ 
\hline
4.80 & 0.9881 & 1.0020 & 1.3174 & 1.3360\\ 
\hline
5.00 & 0.9937 & 1.0071 & 1.3249 & 1.3427\\ 
\hline
\end{tabular}
\end{center}
\caption{$N = 7, \lambda=10$, $\mu=1$, and $\theta=1$}
\label{delta_approx_table_N_7}
\end{table}

\begin{table}[H]
\begin{center}
\begin{tabular}{ |c | c | c | c | c| }
\hline
 $\Delta$ & Nonlinear (1) & Linear (1) & Nonlinear (2) & Linear (2) \\ 
\hline
3.80 & 0.3872 & 0.5646 & 0.4840 & 0.7057\\ 
\hline
4.00 & 0.4136 & 0.5758 & 0.5170 & 0.7197\\ 
\hline
4.20 & 0.4342 & 0.5850 & 0.5427 & 0.7313\\ 
\hline
4.40 & 0.4505 & 0.5927 & 0.5631 & 0.7409\\ 
\hline
4.60 & 0.4635 & 0.5990 & 0.5794 & 0.7488\\ 
\hline
4.80 & 0.4740 & 0.6042 & 0.5925 & 0.7553\\ 
\hline
5.00 & 0.4824 & 0.6085 & 0.6030 & 0.7606\\ 
\hline
5.20 & 0.4892 & 0.6120 & 0.6115 & 0.7650\\ 
\hline
5.40 & 0.4948 & 0.6149 & 0.6185 & 0.7686\\ 
\hline
5.60 & 0.4993 & 0.6173 & 0.6241 & 0.7716\\ 
\hline
5.80 & 0.5029 & 0.6192 & 0.6287 & 0.7740\\ 
\hline
6.00 & 0.5059 & 0.6208 & 0.6324 & 0.7760\\ 
\hline
6.20 & 0.5084 & 0.6221 & 0.6355 & 0.7776\\ 
\hline
\end{tabular}
\end{center}
\caption{$N = 9, \lambda=10$, $\mu=1$, and $\theta=1$}
\label{delta_approx_table_N_9}
\end{table}

Next we consider various plots of queue length against time in Figures \ref{ODD_PLOTS_1}-\ref{ODD_PLOTS_13} for $N=3, N=5, N=7, $ and $N=9$. In Figure \ref{ODD_PLOTS_1} we plot queue length against time for $N=3$ and we see that two of the queue lengths converge to the smaller amplitude and the other queue converges to the larger amplitude. In this case, $\Delta = 3$ which is sufficiently large for all of the queue lengths to be unstable. An interesting phenomenon that we observed in the $N = 3$ case (but did not see in any of the $N \neq 3$ cases) is that if $\Delta$ is larger than the critical delay but not too large, then only two of the three queues will be unstable. This phenomenon is shown in Figure \ref{ODD_PLOTS_2} where we have that $\Delta = 1.5$ (for reference, the critical delay is approximately 0.619 for the chosen parameters in this case). Additionally, in Figure \ref{ODD_PLOTS_3} we plot the amplitude bars computed from solving the nonlinear system and in Figure \ref{ODD_PLOTS_4} we also plot amplitude bars corresponding to the linear amplitude approximation. 

In Figure \ref{ODD_PLOTS_5} we consider the case where $N = 5$ and we see that two of the queue lengths converge to the larger amplitude while the other three queue lengths converge to the smaller queue length. In Figure \ref{ODD_PLOTS_6} we plot the amplitude bars computed from solving the nonlinear system and in Figure \ref{ODD_PLOTS_7} we also plot amplitude bars corresponding to the linear approximation.

In Figure \ref{ODD_PLOTS_8} we consider the case where $N = 7$ and we see that three of the queue lengths converge to the larger amplitude while the other four queue lengths converge to the smaller queue length. In Figure \ref{ODD_PLOTS_9} we plot the amplitude bars computed from solving the nonlinear system and in Figure \ref{ODD_PLOTS_10} we also plot amplitude bars corresponding to the linear approximation.

In Figure \ref{ODD_PLOTS_10} we consider the case where $N = 9$ and we see that four of the queue lengths converge to the larger amplitude while the other five queue lengths converge to the smaller queue length. In Figure \ref{ODD_PLOTS_11} we plot the amplitude bars computed from solving the nonlinear system and in Figure \ref{ODD_PLOTS_12} we also plot amplitude bars corresponding to the linear approximation.

While the value of $\Delta$ used for these figures varied (intentionally, to get plots that are easier to qualitatively examine), we see that the figures support the observation made from Tables \ref{delta_approx_table_N_3}-\ref{delta_approx_table_N_9} that the linear approximation is less accurate when $N$ is larger (noting that we used values of $\Delta$ greater than or equal to those used for smaller values of $N$).



\begin{figure}[ht]
\includegraphics[scale=.65 ]{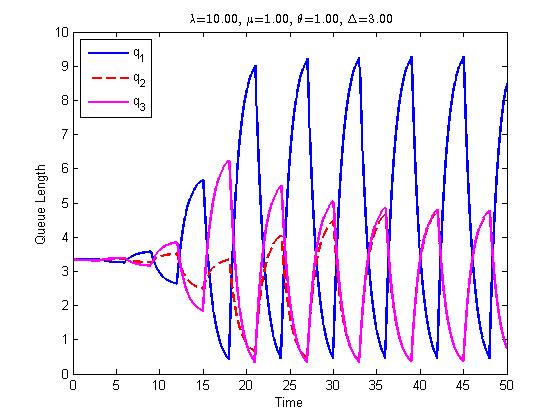}
\caption{$N=3$}
\label{ODD_PLOTS_1}
\end{figure}

\begin{figure}
\includegraphics[scale=.65 ]{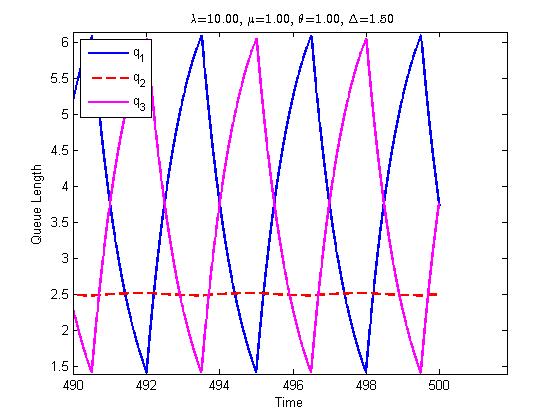}
\caption{Strange case in $N=3$ where one queue length decays}
\label{ODD_PLOTS_2}
\end{figure}

\begin{figure}
\includegraphics[scale=.65 ]{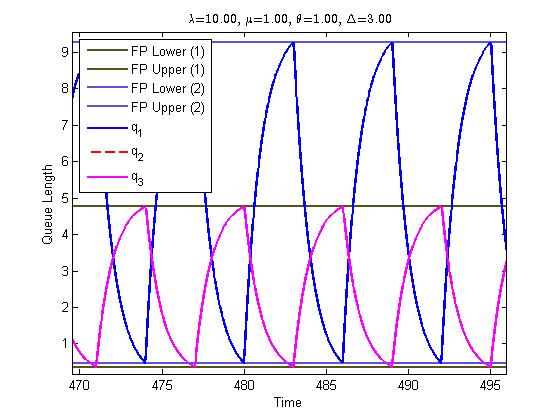}
\caption{$N=3$}
\label{ODD_PLOTS_3}
\end{figure}

\begin{figure}
\includegraphics[scale=.65 ]{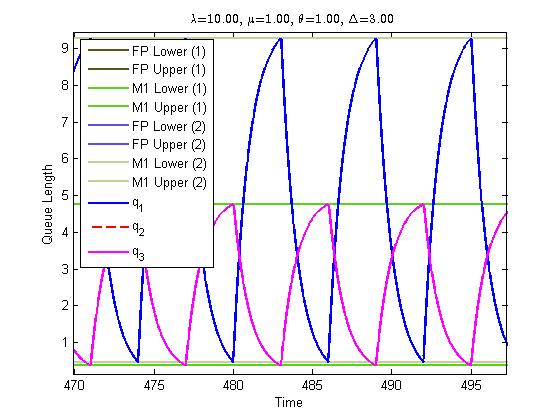}
\caption{$N=3$}
\label{ODD_PLOTS_4}
\end{figure}


\begin{figure}
\includegraphics[scale=.65 ]{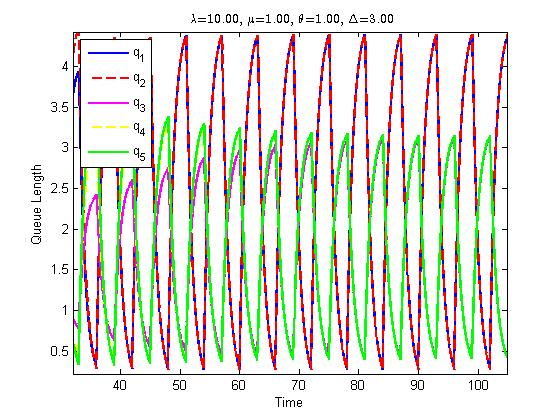}
\caption{$N=5$}
\label{ODD_PLOTS_5}
\end{figure}

\begin{figure}
\includegraphics[scale=.65 ]{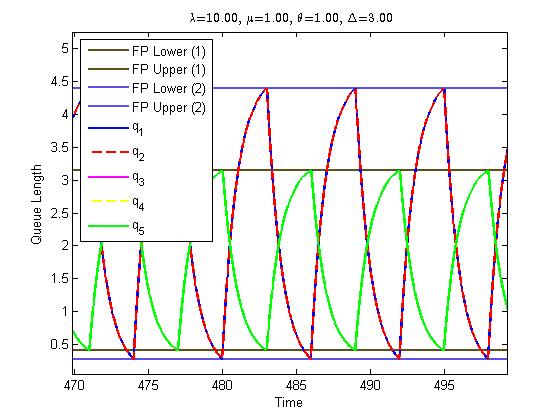}
\caption{$N=5$}
\label{ODD_PLOTS_6}
\end{figure}

\begin{figure}
\includegraphics[scale=.65 ]{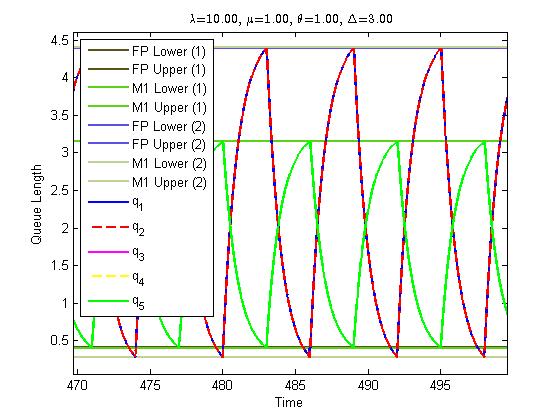}
\caption{$N=5$}
\label{ODD_PLOTS_7}
\end{figure}


\begin{figure}
\includegraphics[scale=.65 ]{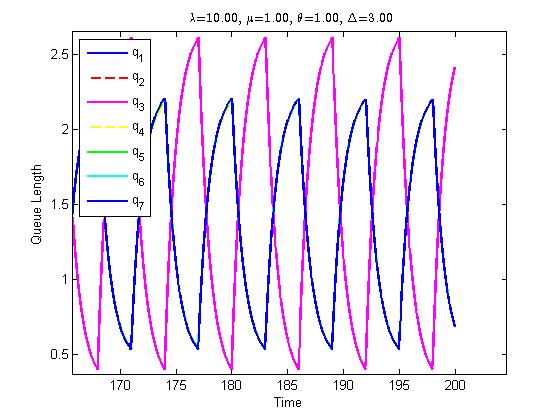}
\caption{$N=7$}
\label{ODD_PLOTS_8}
\end{figure}

\begin{figure}
\includegraphics[scale=.65 ]{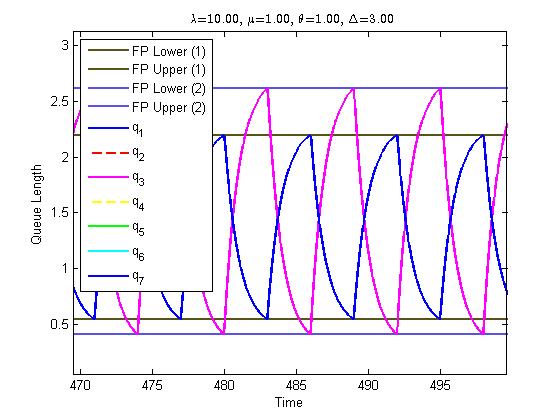}
\caption{$N=7$}
\label{ODD_PLOTS_9}
\end{figure}

\begin{figure}
\includegraphics[scale=.65 ]{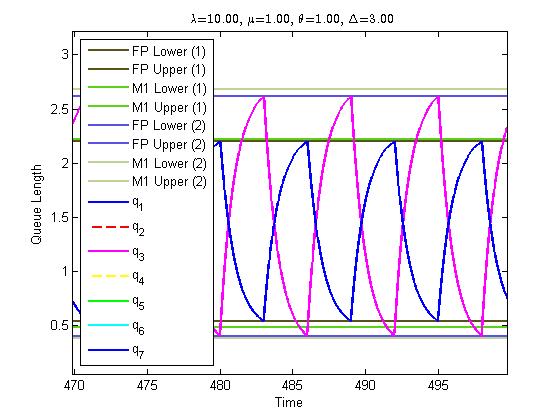}
\caption{$N=7$}
\label{ODD_PLOTS_10}
\end{figure}


\begin{figure}
\includegraphics[scale=.65 ]{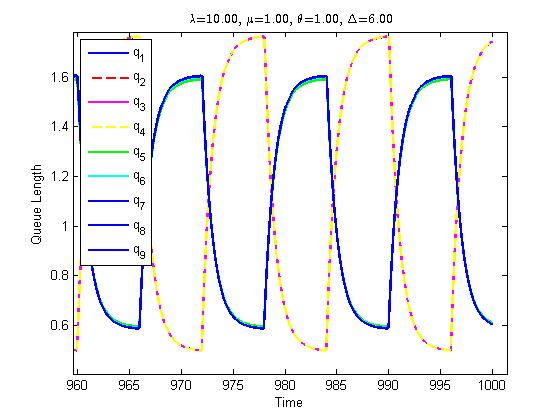}
\caption{$N=9$}
\label{ODD_PLOTS_11}
\end{figure}

\begin{figure}
\includegraphics[scale=.65 ]{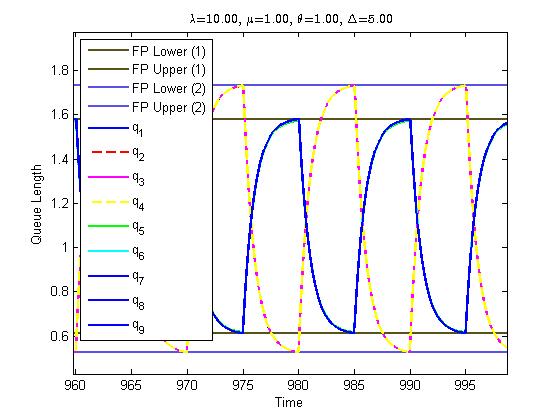}
\caption{$N=9$}
\label{ODD_PLOTS_12}
\end{figure}

\begin{figure}
\includegraphics[scale=.65 ]{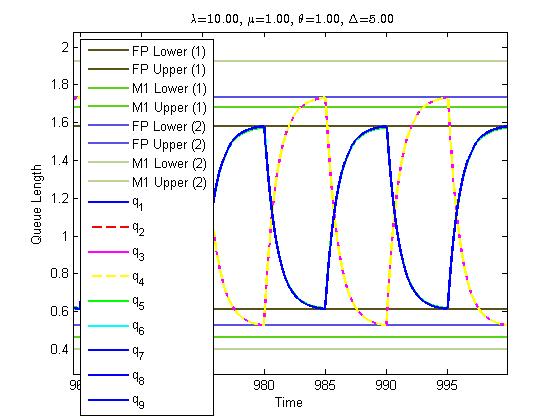}
\caption{$N=9$}
\label{ODD_PLOTS_13}
\end{figure}

\section{Conclusion and Further Research}
\label{conclusion}

In this paper we analyze a system of functional differential equations which models a queueing system with updating information. In particular, we found out how to compute the amplitude of the queue length oscillations that occur after the system experiences a Hopf bifurcation. The methods used in our work differ from those used in the existing literature because the FDE system we consider uses a non-stationary delay and thus techniques used in the existing literature, such as Lindstedt's method which uses asymptotic analysis to approximate the amplitude of oscillations, do not apply to our problem. Indeed, we were able to find a fixed-point equation in the two-dimensional case which can be solved to allow us to exactly compute the amplitude of the oscillatons. Extensions of this fixed-point equation were obtained in the $N$-dimensional case for $N > 2$ as well, where we separately considered the cases for which $N$ is odd or even. While this fixed-point equation and its $N$-dimensional analogues theoretically provide us with a method for exactly computing the amplitude, we can only solve them numerically in practice. Because of this, we developed closed-form approximations of the amplitude via Taylor expansions which we numerically tested the accuracy of for several values of $\Delta$.

There are several potential extensions for future research on this topic. One possible extension could be to consider a choice model that depends on different information, such as information about the time derivative of queue lengths evaluated at the most recent update time. Exploring such an extension could give more insight into how providing customers with different types of information impacts the dynamics of updating queueing systems. Another possible direction for future work could be to try to analyze updating queueing systems that use a time-varying arrival rate instead of the constant arrival rate used in our model. An interesting observation that we made in the three-dimensional case is that one of the three queue lengths had an amplitude that decayed to zero for some values of $\Delta$ that were larger than the critical delay for this system and that this third queue length will approach a nonzero amplitude if $\Delta$ is sufficiently increased past these values. Therefore, there appears to be a second "critical" value of $\Delta$ at which the stability of the third queue changes. Finding a way to compute this second critical value as well as better understanding why we only observed this phenomenon in the three-dimensional case as opposed to higher-dimensional cases would be interesting problems to consider. 

\bibliographystyle{plainnat}
\bibliography{updating_amplitude}

\end{document}